\documentclass[11pt]{amsart}

\usepackage{amssymb}
\usepackage{bm}
\usepackage{graphicx}
\usepackage{psfrag}
\usepackage[dvipsnames]{xcolor}
\usepackage{enumerate}
\usepackage{multirow}
\usepackage{url}

\usepackage{subcaption}
\usepackage{hyperref}
\usepackage{comment}
\usepackage{algpseudocode}
\usepackage{mathtools}

\newcommand{\excise}[1]{}

\newtheorem{theorem}{Theorem}[section]
\newtheorem{lemma}[theorem]{Lemma}

\newtheorem{corollary}[theorem]{Corollary}
\newtheorem{proposition}[theorem]{Proposition}

\newtheorem{maintheorem}{Main Theorem}

\theoremstyle{definition}

\newtheorem{example}[theorem]{Example}
\newtheorem{remark}[theorem]{Remark}
\newtheorem{definition}[theorem]{Definition}

\counterwithin{theorem}{section}

\renewcommand\>{\rangle}

\newcommand\CC{\mathbb{C}}

\newcommand\NN{\mathbb{N}}

\newcommand\QQ{\mathbb{Q}}
\newcommand\RR{\mathbb{R}}

\newcommand\ZZ{\mathbb{Z}}

\DeclareMathOperator\convhull{conv} 

\DeclareMathOperator\aff{aff}
\DeclareMathOperator\aw{aw}
\DeclareMathOperator\AR{AR}
\DeclareMathOperator\gaps{gaps}

\title[An arithmetic measure of width for convex bodies]{An arithmetic measure of width for convex bodies}

\begin{document}

\author[De Loera]{Jes\'us A. De Loera}
\address{Mathematics Department\\University of California Davis\\Davis, CA 95616}
\email{deloera@math.ucdavis.edu}

\author[Marsters]{Brittney Marsters}
\address{Mathematics Department\\University of California Davis\\Davis, CA 95616}
\email{bmmarsters@ucdavis.edu}

\author[O'Neill]{Christopher O'Neill}
\address{Mathematics Department\\San Diego State University\\San Diego, CA 92182}
\email{cdoneill@sdsu.edu}

\makeatletter
\@namedef{subjclassname@2020}{\textup{2020} Mathematics Subject Classification}
\makeatother


\date{\today}

\begin{abstract}

We introduce the arithmetic width of a convex body, defined as the number of distinct values a linear functional attains on the lattice points within the body. Arithmetic width refines lattice width by detecting gaps in the lattice point distribution and always provides a natural lower bound. We show that for large dilates of a convex body, the attained values form an arithmetic progression with only a bounded number of omissions near the extremes. For rational polytopes, we show that the arithmetic width grows eventually quasilinearly in the dilation parameter, with optimal directions reoccurring periodically. Lastly, we present algorithms to compute the arithmetic width. These results build new connections with discrete geometry, integer programming, and additive combinatorics.  
\end{abstract}

\maketitle
\vspace{5mm} 

\section{Introduction}
Convex bodies admit several different width measurements, most notably the Euclidean and lattice (or integer) widths. These invariants have proven useful in many areas of mathematics. In integer optimization, lattice width has been instrumental in the development of algorithms, such as Lenstra's algorithm, shortest vector computation, and other lattice algorithms \cite{lenstra1983integer,nguyen2011lllsurvey}. Moreover, in computer graphics, lattice width is used to describe straightness in digital spaces \cite{Feschet_digitalspaces}.  In convex geometry, lattice width appears in Borsuk‐type partitioning and lattice‐based sphere‐packing and covering problems \cite{barvinok2002course,lovasz1989geometry}. 

Our paper proposes a new alternative notion of width for convex bodies, called the \emph{arithmetic width}, that seeks to address limitations of lattice width and provides alternative lower bounds for estimation. Before we introduce our new definition, let us briefly recall established notions of width. The lattice width of a convex body $K\subseteq\RR^d$ is
\[
w(K)=\displaystyle\min_{c\in \ZZ^d\setminus\{0\}} w_c(K),
\qquad \text{where} \qquad
w_c(K) = \max_{x,y\in K}c^T(x-y).
\]
Geometrically, the lattice width measures the minimum of the Euclidean distance between two hyperplanes with normal $c\in\ZZ^d$, scaled by the length $||c||$ of that integer vector. In~some sense, the lattice width measures the thinness of the convex body in integer directions. In~Figure~\eqref{fig: lattice width examples: nonintegral}, each of the dashed lines contain a lattice point of the polytope. However, the solid line contains no lattice points. If $K$ ``slips through'' the integer lattice -- i.e.,\ if there are gaps between lattice-supporting hyperplanes that intersect $\ZZ^d \cap K$ -- then $w(K)$ remains blind to those gaps. In~contrast, the arithmetic width is a finer invariant than the integer or Euclidean widths as it accounts for these missing lattice points in our polytope in fixed directions.

\begin{figure}[t]
    \centering
    \begin{subfigure}[t]{0.40\textwidth}
        \includegraphics[width=\linewidth]{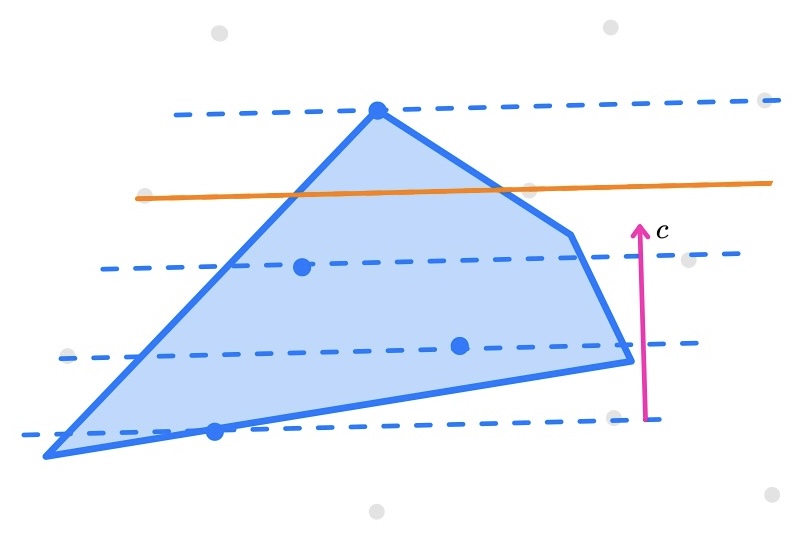}
        \caption{}
        \label{fig: lattice width examples: nonintegral}
    \end{subfigure}
    \hspace{0.02\textwidth}
    \begin{subfigure}[t]{0.40\textwidth}
    \includegraphics[width=\linewidth]{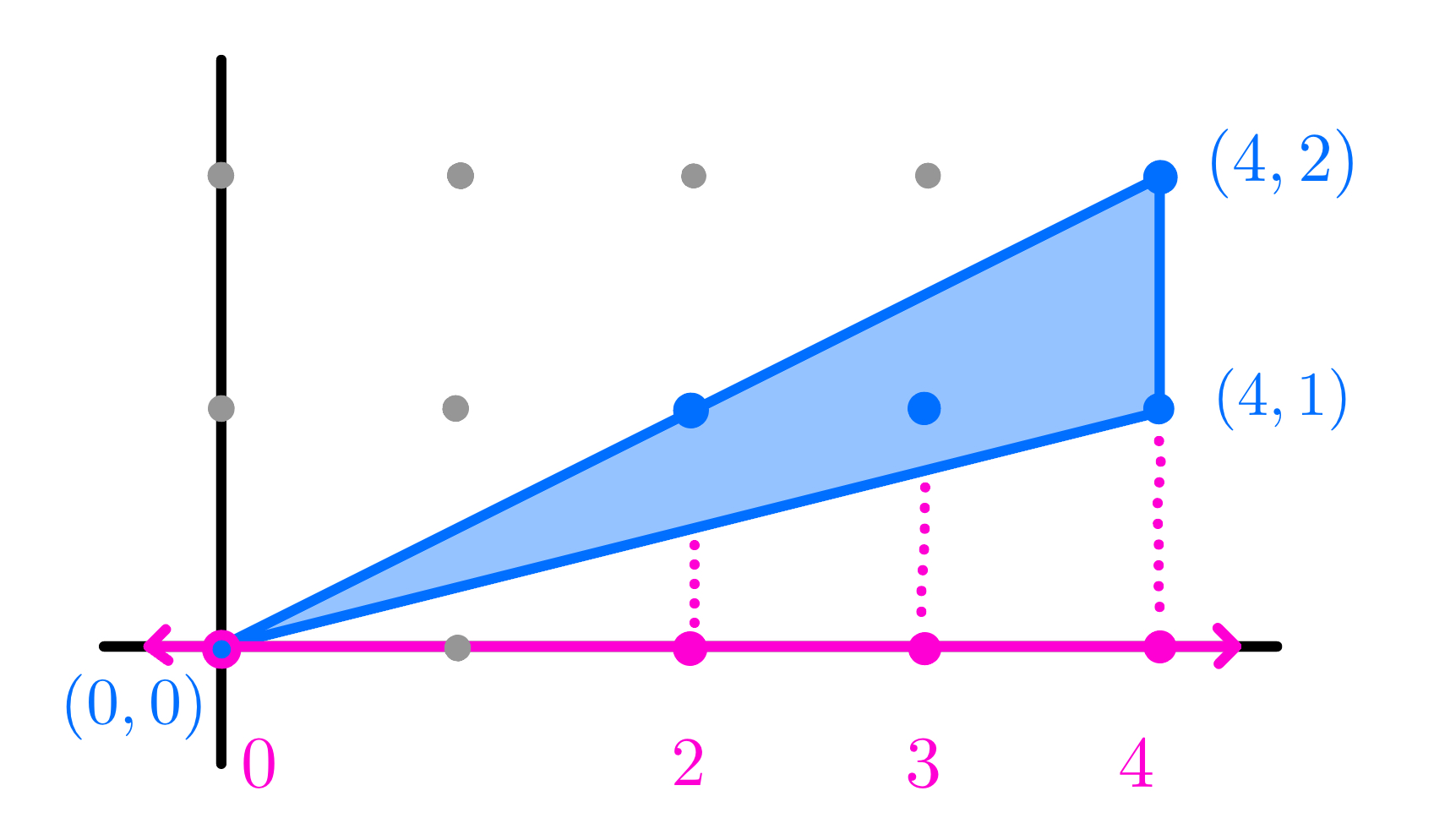}
        \caption{}
        \label{fig: lattice width examples: integral}
    \end{subfigure}
    \caption{A convex body $K$ depicted with a hyperplane contributing to lattice width but containing no lattice points in $K$ (left), and the polytope~$P$ in Example~\ref{ex:gaps} (right).}
    \label{fig: lattice width examples}
\end{figure}

\begin{definition}
    Fix a nonzero $c \in \ZZ^d$. The \emph{hyperplane cover of $K\cap\ZZ^d$ in direction $c$} is
    \[h_c(K) = \{H : H\hspace{1mm}\text{is a hyperplane with normal}\hspace{1mm} c,\hspace{1mm} H\cap K\cap\ZZ^d\neq\emptyset\}.\] The \emph{arithmetic width of $K$ in direction $c$} is $\aw_c(K)= |h_c(K)|$. 
\end{definition}

Lattice and arithmetic width satisfy the following inequality: \[\aw_c(K)=|h_c(K)|\leq w_c(K) + 1.\]
Whenever there are gaps, the arithmetic width and lattice width must differ.

Few results are know on lower bounds for lattice width, which makes arithmetic width a valuable new tool in the study of widths \cite{khintchine1948,Mayrhofer_et_al}. Moreover, it is clear that $\aw_c(K)=0$ precisely when there are no lattice points inside $K$. In contrast,  Khinchin’s celebrated flatness theorem~\cite{khintchine1948} guarantees that for every dimension $d$ there exists a positive constant which upper-bounds the lattice width of convex bodies which have no lattice points inside. The~exact value of the so called \emph{flatness constant} is generally not known, even in dimension $d=3$ (see~\cite{codenotti2021generalisedflatnessconstantsframework,ReisRothvoss2023} and references therein), but it is conjectured to be roughly proportional to~$d$.

Lattice and arithmetic widths also differ in ways beyond their sensitivity to empty slices of the convex body. While it is clear that the optimal values of lattice and arithmetic widths can differ, we will see in Example~\ref{ex: aw_vs_lw} that even the optimal directions for these notions of width can differ. Distinguishing properties of arithmetic width and lattice width are summarized in Table~\ref{tab:width_comparison}.

There are other width-adjacent notions with clear connections to the arithmetic width, namely the layer-number and $k$-level of a lattice polytope. In \cite{DezaOnn1995}, while investigating diameters of graphs of lattice polytopes, Deza and Onn introduced the \emph{layer-number in direction~$c$} as the number of distinct values the linear functional $c^Tx$ takes over the vertices of a lattice polytope. They defined the layer-number of $P$ as the minimum over all non-zero directions and showed that there is a constant~$C$ that bounds the layer number of any $d$-dimensional convex polytope by $Cd^2$. As the layer-number only considers vertices in its count, the layer-number serves as a lower bound for the arithmetic width. Similarly, a lattice polytope $P$ is \emph{$k$-level} if every facet-defining linear function of $P$ takes at most $k$ distinct values on the vertices of the polytope. Thus, for any $k$-level polytope, the arithmetic width in facet defining directions is at most $k$. The importance of $k$-level polytopes stems from their connections to optimization. It is well-known that if P is $k$-level, then it admits a spectrahedral lift of polynomial size (for fixed $k$). Notably, being $k$-level is used to give efficient semidefinite programming representations for stable set polytopes of perfect graphs~\cite{Fawzietal2015}.

\begin{table}[t]
    \centering
    \begin{tabular}{ c|c } 
        Lattice Width & Arithmetic Width\\ \hline Real-valued & Integer-valued 
        \\
        Insensitive to gaps &  Sensitive to gaps:\ see Figure~\ref{fig: lattice width examples}
        \\
        Parallel vectors give different lattice widths & Parallel vectors give same arithmetic width
        \\
        Obtained by ``short'' directions of the lattice & The magnitude of the vector is irrelevant
        \\
        $|h_c(K)|\leq w_c(K)+1$ & $|h_c(K)| = aw_c(K)$
        \\
        For $\lambda>0$, $w(\lambda K)=\lambda w(K)$ &  $\aw{
        (\lambda K})$ quasiperiodic:\ see Theorem~\ref{mainthm: eql_non_fixed_directions}
    \end{tabular}
    \smallskip
    \caption{Differences between lattice width and arithmetic width}
    \label{tab:width_comparison}
\end{table}

A complementary perspective comes from additive combinatorics, via the notion of $P$-free sets, inspired by the classical Sidon problem of Erd\"os. 
For a given convex polyhedron $P$, one can ask what is the largest subset of integers~$A$ such that $P\cap (\ZZ^d\cap A^d)=\emptyset$? In other words, a subset $A\subseteq \{0,\dots, n\}$ is called $P$-free if no lattice point of $P$ has all its coordinates in $A$. This question generalizes classical questions in additive combinatorics related to Sidon and Salem-Spencer sets~\cite{Balogh28052023, salem-spencer-sets}. For example, consider the polyhedron $P\subseteq \RR^3$ defined by $x_1+x_3=2x_2$ and $x_1,x_2,x_3\geq 0$. In this case, a set $A$ is $P$-free if it contains no three-term arithmetic progression. This combinatorial perspective motivates the us to study which integer values appear as coordinates, informing possibilities for the $P$-free sets $A$. 

To make the above precise, we introduce the notion of the arithmetic range. 

\begin{definition}
    Let $K\subseteq \mathbb{R}^d$ be a convex body and let $\ell_c: \mathbb{Z}^d\to \mathbb{Z}$ be a linear functional given by $\ell_c(x)=c^Tx$. The \emph{arithmetic range} of $K$ with respect to the direction $\ell_c$ is 
    \[
    \AR_{c}(K)=\{n\in\mathbb{Z}: n=\ell_c(x) \text{ for some } x\in K\cap\mathbb{Z}^d\}.
    \] 
\end{definition}

\begin{example}
    \label{ex:gaps}
    Let $P=\convhull((0,0),(4,1),(4,2))\subset\RR^2$ as in Figure~\ref{fig: lattice width examples: integral}, and let $c=(0,1)$. Then, we have \[\AR_{c}(P)=\{0,2,3,4\}.\] In particular, notice that $1$ is not in the arithmetic range. 
\end{example}

We will soon prove (see Proposition~\ref{prop: aw_properties}\eqref{prop: aw_properties; equivalence}) that \[\aw_{c}(K):=|AR_c(K)|.\]

\subsection{Main results}
The main results of this paper comprise of four main theorems. The first theorem is a structure theorem for the arithmetic range:\ for sufficiently large dilates of a convex body, the arithmetic range has a highly regular structure, forming an ``almost arithmetic progression'' - an arithmetic progression except for finitely many exceptions near the extremes of the set, a notion which arises frequently in the settings of combinatorial algebra and number theory~\cite{nonuniquefactorization}. 

\begin{maintheorem}
\label{mainthm: AR structure}
    Let $K$ be a convex body in $\mathbb{R}^d$ and let $\ell_c:\ZZ^d\to\ZZ$ be a linear functional. 
    There exist constants $t$, $t'$, and $\lambda$ dependent only on $K$ such that, for all $n \gg 0$, the set $\AR_{c}(nK)\cap [nm+t, nM-t']$ is an arithmetic progression of step size $\lambda$ where $m=\min_{x\in K}\ell_c(x)$ and $M=\max_{x\in K}\ell_c(x).$
\end{maintheorem}

\begin{figure}[t]
    \centering
    \includegraphics[width=0.65\linewidth]{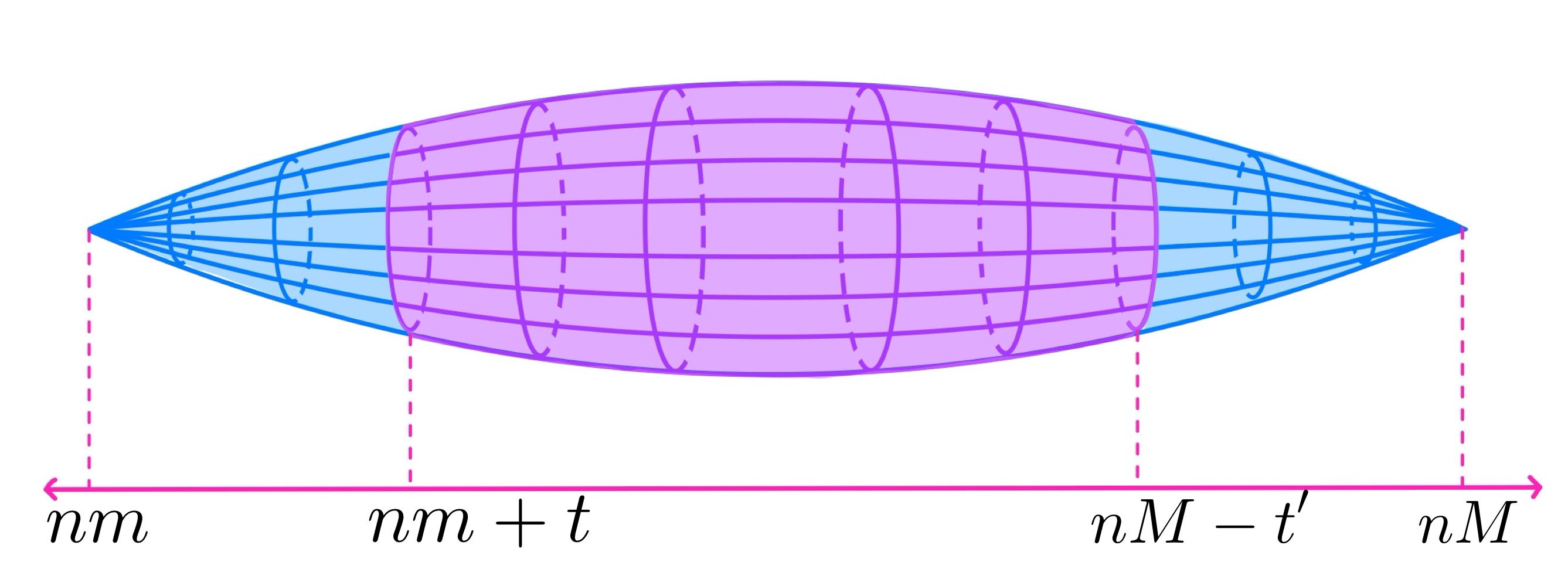}
    \caption{The constants $n$, $t$, and $t'$ from Main Theorem~\ref{mainthm: AR structure}.}
    \label{fig:AAP_constants}
\end{figure}

The study of arithmetic progressions in dilates of convex bodies was started by Freiman~\cite{freiman} and other variations of this problem were popularized by Ruzsa~\cite{ruzsafreiman} and later by Tao~\cite{taofreiman}. The behavior in our variation of this problem mirrors the eventually quasipolynomial growth of both Hilbert functions of graded algebras and of Ehrhart functions of rational polytopes \cite{BrunsHerzogHilbert,MillerSturmfels,StanleyCCA}. In~particular, just as the eventual quasipolynomial behavior of the classical Hilbert function detects structural stabilization in graded pieces of high degree, the arithmetic ranges of sufficiently large dilates of a convex body also exhibit structural stabilization.
Additionally, a number of results in additive combinatorics and factorization theory state that, under mild assumptions, sets of lengths of semigroup elements form almost arithmetic progressions~\cite{sets_of_length}.  Main Theorem~\ref{mainthm: AR structure} establishes an analogous property for arithmetic ranges of convex bodies in such a way that results concerning sets of lengths for numerical semigroups are obtained as a direct corollary (Remark~\ref{rmk: structure theorem for sets of length}). 

Having established structural properties for arithmetic width of dilates of a convex body, it is natural to ask how the arithmetic width grows as we dilate. 
In contrast, lattice width grows linearly in the dilation parameter. Our second result continues to mirror the behavior of Hilbert and Ehrhart functions:\ as a rational polytope is dilated, the size of the arithmetic range grows quasilinearly with finitely many exceptions. Remarkably, this behavior holds not only with a fixed linear functional, but persists when we minimize. 

\begin{maintheorem}
\label{mainthm: eql_non_fixed_directions}
    Let $P\subset\RR^d$ be a rational polytope. 
    \begin{enumerate}[(a)]
        \item 
        For a fixed linear functional $\ell_c:\ZZ^d \to \ZZ$, the arithmetic width $aw_{\ell}(nP)$ is an eventually quasilinear function of $n$ for $n\in\ZZ_{\geq 0}$.
        \item 
        The arithmetic width $aw(nP)$ is an eventually quasilinear function of $n$ for $n\in\ZZ_{\geq 0}$.
    \end{enumerate}
\end{maintheorem}

The previous two theorems describe the structural properties of the arithmetic ranges and the asymptotic behavior of the arithmetic width. From an algorithmic perspective, it is equally important to determine whether these quantities can be computed efficiently. Note that computing the lattice width is already $\mathsf{NP}$-hard. It is even hard to approximate; for instance, computing the lattice width of a centrally symmetric rational polytope from its vertex description can be reduced to computing the shortest vector in a lattice \cite{Aroraetal1997,vanEmdeBoas1981}. However, the problem is solvable in polynomial time in fixed dimension $d$; the statement ``$w(P)\le t$?'' can be reduced to a mixed-integer program with only $d$ integer variables and then solved efficiently using Lenstra's or Barvinok's algorithms~\cite{barvinokpommersheim,lenstra1983integer}. We remark that the computation of lattice width has been discussed or applied by many researchers; see, e.g., \cite{CharrierBuzerFeschet2009,dadush2012integer,ReisRothvoss2023,MicciancioGoldwasser2002} and the references therein. We next give analogous algorithmic results for computing the arithmetic width of rational polytopes. 

\begin{maintheorem} 
\label{mainthm:awu-polytime}
    Let $P\subset\RR^d$ be a rational polytope and let $L$ be the input bit-size of $P$.
    \begin{enumerate}[(a)]
        \item In fixed dimension, given a nonzero integer vector $c\in\ZZ^d\setminus\{0\}$, there is an algorithm to compute $\aw_c(P)$ that is polynomial time in~$L$ and the input bit-size of $c$.
        \item[(b)] There is a finite set of directions $T$ that serve as a sufficient test set to compute $\aw(P)$ (i.e., $\aw(P)=\aw_t(P)$ for some $t\in T$). Thus, in fixed dimension, there is an algorithm to compute the arithmetic width of $P$ that is single-exponential time in~$L$.
    \end{enumerate}
\end{maintheorem}

In Main Theorem~\ref{mainthm:two-examples} part~(a) we show another difference between the arithmetic and lattice widths:\ even the directions they attain their minimum values may differ.  
In contrast to Main Theorem~\ref{mainthm: AR structure}, which establishes a structural property for arithmetic ranges of large dilations of convex bodies, part~(b) of Main Theorem~\ref{mainthm:two-examples} demonstrates the opposite extreme:\ any finite set of integers can be realized as the arithmetic range of some simplex.

\begin{maintheorem}
\label{mainthm:two-examples}
The following hold.  
\text{ }
    \begin{enumerate}
        \item[(a)] The directions that obtain the minimum lattice width and the minimum arithmetic width can differ.
        \item[(b)] Any finite set of integers occurs as the arithmetic range of some rational simplex.
    \end{enumerate}
\end{maintheorem}



\section{Arithmetic width:\ properties and examples}

To provide intuition and necessary tools to prove our main results, we collect in this short section several basic properties of arithmetic width and give further examples. These results clarify how arithmetic width behaves under common geometric operations, demonstrate the variety of possible arithmetic ranges, and continue the discussion of the distinction between arithmetic width and lattice width. In all that follows, we denote $[n]=\{1,\dots, n\}\subseteq \ZZ_{>0}$.

\begin{proposition}
\label{prop: aw_properties}
Let $K\subset\RR^d$ be a convex body. 
\begin{enumerate}[(a)]
    \item \label{prop: aw_properties; translation}
    The arithmetic width $\aw_c(K)$ for a nonzero integer vector $c$ is invariant under integer translations of $K$. Specifically, if $t\in\ZZ^d$, then $\aw_c(K)= \aw_c(t+K)$.

    \item \label{prop: aw_properties; parallel}
    Let $c_1$ and $c_2$ be nonzero parallel integer vectors. Then, $\aw_{c_1}(K)=\aw_{c_2}(K)$.
    
    \item \label{prop: aw_properties; equivalence}
    For a nonzero integer direction $c$, 
    \[\aw_c(K):= |\AR_c(K)|= h_c(K).\]
    
    \item \label{prop: aw_properties; maxwidth}
    The maximal arithmetic width of $K$ is $|K\cap\ZZ^d|$. 
\end{enumerate}
\end{proposition}

\begin{proof}
Fix a nonzero integer vector $c$. As $|K\cap\ZZ^d|= |(K+t)\cap\ZZ^d|$ and that $\ell_c(x_1)=\ell_c(x_2)$ if and only if $\ell_c(x_1 +t)=\ell_c(x_2+t)$ for $x_1,x_2\in K\cap\ZZ^d$, we have claim~(a): \[\aw_c(K)= |\AR_c(K)| = |\AR_c(K+t)| =  \aw_c(K+t).\]

Now fix two non-zero, parallel, integer vectors $c_1$ and $c_2$. Without loss of generality, let $c_1> c_2$, thus there is some $\lambda\in\ZZ_{\neq 0}$ where $c_1=\lambda c_2$. For $x\in K\cap\ZZ^d$, as $\ell_{c_1}(x)=\lambda\ell_{c_2}(x)$, we see that $\AR_{c_1}(K)=\lambda\AR_{c_2}(K)$.  This gives us claim~(b).    

We can partition $K\cap\ZZ^d$ by their images under $\ell_c$. Letting $K_j=\{x\in K\cap\ZZ^d: \ell_c(x)=j\}$, we have that \[K\cap\ZZ^d = \bigcup_{j\in \AR_c(K)} K_j.\] To each $K_j$, there is a corresponding $H\in h_c(K)$ given by $c^Tx=j$. This proves claim~(c). 

Now, for claim~(d), consider the moment curve $g(x)=(1, x, \dots, x^{d-1})$ and let $n=|P\cap\ZZ^d|$. For each pair of distinct lattice points $x,y\in P\cap\ZZ^d$, consider the polynomial given by $f_{x,y}(z)=g(z)^T(x-y)$. For each pair $x,y$, this polynomial is of degree at most $d-1$ and thus has at most $d-1$ roots. Let $Z$ be the union of the sets of all roots of each possible $f_{x,y}(z)$. Thus, $|Z| \leq (d-1)\binom{n}{2}$. Choosing any integer $z\in\{0, 1, \dots, (d-1)\binom{n}{2}+1\}\setminus Z$ gives us that $\ell_{g(z)}(x)\neq \ell_{g(z)}(x)$ for all distinct $x,y\in P\cap\ZZ^d$. Thus, no two lattice points of $K$ map to the same place in the arithmetic range, giving us our final claim. 
\end{proof}

To conclude this section, we provide two examples: the first makes it clear that the arithmetic and lattice widths are truly distinct notions of width; and the second allows us to realize arbitrary finite sets as arithmetic ranges of certain simplices.

\begin{example}
\label{ex: aw_vs_lw}
        Let $P=\convhull((0,0),(1,0), (\tfrac{3}{5},\tfrac{14}{5})).$ Since $P\cap\ZZ^2=\{(0,0),(1,0)\}$, the optimal arithmetic width is obtained in the directions $c^*=\pm(0,1)$ that collapse $(0,0)$ and $(1,0)$, giving us $\aw(P)=1$. The lattice width in this direction is \[w_{\pm(0,1)}(P) = \tfrac{14}{5}.\] Additionally, for the lattice width, we have that $w(P)=1$. However, the only directions that obtain the lattice width for $P$ are $\pm(1,0)$. From this, we can conclude that the directions where the optimal arithmetic width is obtained can differ from the directions where the optimal lattice width is obtained. In particular, the sets where the respective minimal widths are obtained can form disjoint sets.
\end{example}

\begin{example}
\label{ex: any finite set is an arithmetic range}
    Sets arising as the images of lattice points under a linear functional appear in several areas of combinatorics and algebra \cite{weighted_ehrhart, MR1992831, Secant}. One such area is in the study of numerical semigroups, that is, sets of the form
    \[S=\<g_1,\dots,g_k\> = \{n\in \ZZ_{\geq 0}: n= a_1g_1+\dots+a_kg_k, \, a_i\in\ZZ_{\geq 0} \text{ for } i\in [k]\}.\]
    The numerical semigroup polytope of $S$ is 
    \[P_S=\convhull\big(\tfrac{1}{g_1}e_1,\dots, \tfrac{1}{g_k}e_k\big),\]
    where $e_i$ is the standard $i$-th basis vector. The lattice points in $nP_S$ correspond to the \emph{factorizations} of $n$ in $S$; that is, each $(x_1,\dots, x_k)\in nP_S\cap\ZZ^k$ corresponds to a nonnegative integer solution to $n=g_1x_1+\dots+g_kx_k$.  Applying the linear functional $c = (1,\dots,1)$ to such a lattice point yields its \emph{factorization length}, and as such, $\AR(nP_S)$ coincides with the set of lengths of $n \in S$.  As stated in \cite[Theorem~3.3]{sets_of_length}, any finite nonempty set $L \subseteq \NN_{\geq 2}$ occurs as the set of lengths of some element $n$ of some numerical semigroup $S$.  By translating $nP_S$ appropriately and applying Proposition~\ref{prop: aw_properties}\eqref{prop: aw_properties; translation}, we conclude that every finite subset of $\ZZ$ occurs as the arithmetic range of some rational simplex.  

\end{example}


\begin{proof}[Proof of Main Theorem~\ref{mainthm:two-examples}]
This follows from Example~\ref{ex: aw_vs_lw} and Example~\ref{ex: any finite set is an arithmetic range}.  
\end{proof}

\section{The arithmetic range of convex bodies}

In this section, we prove our first main result:\ for sufficiently large dilates of convex bodies, the arithmetic range is ``almost" an arithmetic progression.  We first make this notion precise.  

\begin{definition}
\label{def: almost arithmetic progression}
    We say a finite, nonempty set $S\subset\ZZ$ is an \emph{almost arithmetic progression} with step size $\lambda$, left bound $t$, and right bound $t'$ if
    \[
    S = \{m, m+\lambda, \dots, M-\lambda, M\}
           \setminus ( A \cup A')
    \] 
    for some $A \subseteq [m,m+t] \cap (\lambda \ZZ + m)$ and $A' \subseteq [M-t',M] \cap (\lambda \ZZ + m)$, 
    where $m=\min(S)$ and $M=\max(S)$.  
    Said another way,
    \[ (\lambda \ZZ + m) \supseteq S \supseteq [m+t, M-t'] \cap (\lambda \ZZ + m).
    \]
    We call 
    \[
    \gaps(S) = A \cup A',
    \qquad
    \gaps_l(S) = A, 
    \qquad \text{and} \qquad
    \gaps_r(S) = A'
    \]
    the \emph{gaps}, \emph{left gaps}, and \emph{right gaps} of $S$, respectively.   
\end{definition}

\begin{example}
\label{ex: almost arithmetic progression}
Returning to Example~\ref{ex: any finite set is an arithmetic range}, for $S = \langle 22,79,91,190 \rangle$ and $n = 4180$, we have
\[
\AR(nP_S) = \{22\} \cup \{28, 31, 34, \ldots, 172\} \cup \{178, 190\},
\]
which is an almost arithmetic progression with step size $\lambda = 3$, left bound $t = 3$, and right bound $t' = 15$.  
It turns out for any $n \in S$, the arithmetic range $\AR(nP_S)$ is an almost arithmetic progression with identical step size and left and right bounds.  This phenomenon is known to occur for any numerical semigroup; see the discussion in Remark~\ref{rmk: structure theorem for sets of length}.  
\end{example}

In what follows, we say $K\subseteq\RR^d$ is a convex body if it is a compact, convex set. We do not require that $K$ has a non-empty interior, allowing $\dim(K) < d$. Additionally, let $B_r(c)\subset\RR^d$ be the Euclidean ball of radius $r>0$ centered at $c\in\RR^d$.

\begin{theorem}
\label{thm: AR structure}
    Let $K$ be a convex body in $\mathbb{R}^d$ whose affine span contains a lattice point, and let $\ell_c:\ZZ^d\to\ZZ$ be a linear functional. For $n\gg 0$, the arithmetic range of $nK$ is an almost arithmetic progression whose step size, left bound, and right bound depend only on $K$. 
\end{theorem}

\begin{proof}
Let $m=\min_{x\in K}\ell_c(x)$ and $M=\max_{x\in K}\ell_c(x)$.  
Let $\mathcal{L} = \aff(K) \cap \mathbb{Z}^d$ and $J = K \cap \mathcal{L}$. Since $\aff(K)$ contains a lattice point, by Proposition~\ref{prop: aw_properties}\eqref{prop: aw_properties; translation} we may assume $\aff(K)$ contains the origin, so in particular $\mathcal L$ is a sublattice of $\ZZ^d$. Fix $\lambda \in \ZZ$ such that $\lambda\ZZ = \ell_c(\mathcal L)$.  We claim 
\[
nJ \cap \mathbb{Z}^d = nK \cap \mathbb{Z}^d
\]
for all $n \in \mathbb{Z}_{\geq 0}$. The inclusion $nJ \cap \mathbb{Z}^d \subseteq nK \cap \mathbb{Z}^d$ follows immediately from $J \subseteq K$. For the other direction, suppose $z \in nK \cap \mathbb{Z}^d$. Then $z = nx$ for some $x \in K$. As $z \in \aff(K) \cap \mathbb{Z}^d = \mathcal{L}$, we have $z \in \aff(\mathcal{L})$. If $x \notin \aff(\mathcal{L})$, then $nx \notin \aff(\mathcal{L})$, a contradiction. Thus, $x \in \aff(\mathcal{L})$, so $x \in J$ and $z \in nJ$, proving the claimed equality. When considering $nK \cap \mathbb{Z}^d$, we may therefore replace $K$ with $J = K \cap \mathcal{L}$. Thus, we may assume without loss of generality that the affine span of $K$ has the same dimension as the rank of its lattice points, that is, $\dim(\aff(K)) = \mathrm{rank}(\aff(K) \cap \mathbb{Z}^d)$. 

Without loss of generality, we can assume $c$ is primitive, since if $c=\mu c_0$ for $\mu\in \ZZ$ and $c_0\in \ZZ^d$, then $\AR_c(K)= \mu\AR_{c_0}(K)$. Let $H=\ker(\ell_c)\cap \aff(K)$.  If $H = \aff(K)$, then $\AR_c(nK) \subseteq \{0\}$ for all $n$, and all desired conclusions immediately follow.  As such, we may assume $\dim H = \dim \aff(K) - 1$.  

Choose $r>0$ such that any translate of $B=B_r(0) \cap H$ in $H$ will contain a lattice point. Now, dilate $K$ until $K$ contains a lattice point $x$ such that $\widetilde{B} \cap H \subseteq K$ for $\widetilde{B}= B_r(x) \cap \aff(K)$. Call this dilation factor $N$. Note that $\dim(B)= \dim(\widetilde{B})-1$, and that $\widetilde{B}$ contains a translate of $B$, namely $\widetilde{B} \cap (H+x)$. Let $p_m, p_M\in K$ attain the minimum value $m$ and maximum values $M$ respectively.
Define 
\[t:=c^T(x-Np_m) \quad\text{and}\quad t':=c^T(Np_M-x)\]
as the lengths of the intervals $[\ell_c(x), \ell_c(Np_m)]$ and $[\ell_c(x), \ell_c(Np_M)]$, respectively.
Now, for each $n$, let $\hat{p}_m(n)$ and $\hat{p}_M(n)$ denote the unique points on the line containing $np_m$ and $np_M$ such that 
\[t=\ell_c(\hat{p}_m(n)) - \ell_c(np_m) \quad\text{and}\quad t'= \ell_c(np_M) - \ell_c(\hat{p}_M(n)),\]
i.e., the images of $\hat{p}_m(n)$ and $\hat{p}_M(n)$ are distances $t$ and $t'$ from the images of $np_m$ and $np_M$, respectively.  Note that for $n \gg N$, we have $\hat{p}_m(n), \hat{p}_M(n) \in K$.  More precisely, any point on the line containing $np_m$ in direction $c$ can be written as $x_{\gamma} = np_m +\gamma c$ for $\gamma\in\RR$. Writing the desired $\hat{p}_m(n)$ as $\hat{p}_m(n) =np_m +\gamma c$ for some $\gamma\in\RR$ gives us that 
\[\ell_c(\hat{p}_m(n)) = \ell_c(np_m) + t = \ell_c(np_m +\gamma c)= \ell_c(np_m)+\gamma||c||^2.\]
Solving for $\gamma = \tfrac{t}{||c||^2}$, and applying a similar argument for $\hat{p}_M(n)$, gives us that 
\[\hat{p}_m(n) = np_m +\frac{t}{||c||^2} c
\qquad \text{and} \qquad
\hat{p}_M(n) = np_M - \frac{t'}{||c||^2}c.\]

By B\'ezout's identity, since $c$ is primitive, for any integer $j$, there exists $x\in\ZZ^d$ such that $c^Tx=j$. Thus, for any $j \in \lambda\ZZ$ we have $\ell^{-1}_c(j)\cap\ZZ^d\neq \emptyset$.  In particular, for any $j\in [\ell_c(\hat{p}_m (n)), \ell_c(\hat{p}_M(n))] \cap \lambda\ZZ$, we have that $\ell^{-1}_c(j) \cap \aff(K) = H + w$ for some $w \in \mathcal L$. Moreover, since $K$ is convex and $B_r(\hat{p}_m(n)), B_r(\hat{p}_M(n)) \subseteq K$, we have $B_r(z)\cap\aff(K)\subseteq K$ for any $z\in [\hat{p}_m(n), \hat{p}_M(n)]$. Particularly, if we additionally assume $z\in\ell_c^{-1}(j)$,  then the ball $B_r(z)\cap \ell_c^{-1}(j)\cap\aff(K)$ is a translate of $B$ in $H+w$. As such, by the choice of $r$ and the fact that $w \in \ZZ^d$, we see $B_r(z)\cap \ell_c^{-1}(j) \cap \aff(K)$ contains an integer point in $nK$.  Thus, we have found constants $t$, $t'$, $\lambda$ dependent only on $K$ such that, for $n \gg 0$, the set $AR_{c}(nK)\cap [nm+t, nM-t']$ is an arithmetic progression of step size $\lambda$.  
\end{proof}

\begin{proof}[Proof of Main Theorem~\ref{mainthm: AR structure}]
Apply Theorem~\ref{thm: AR structure}.  
\end{proof}

\begin{figure}[t]
    \centering
    \includegraphics[width=0.6\linewidth]{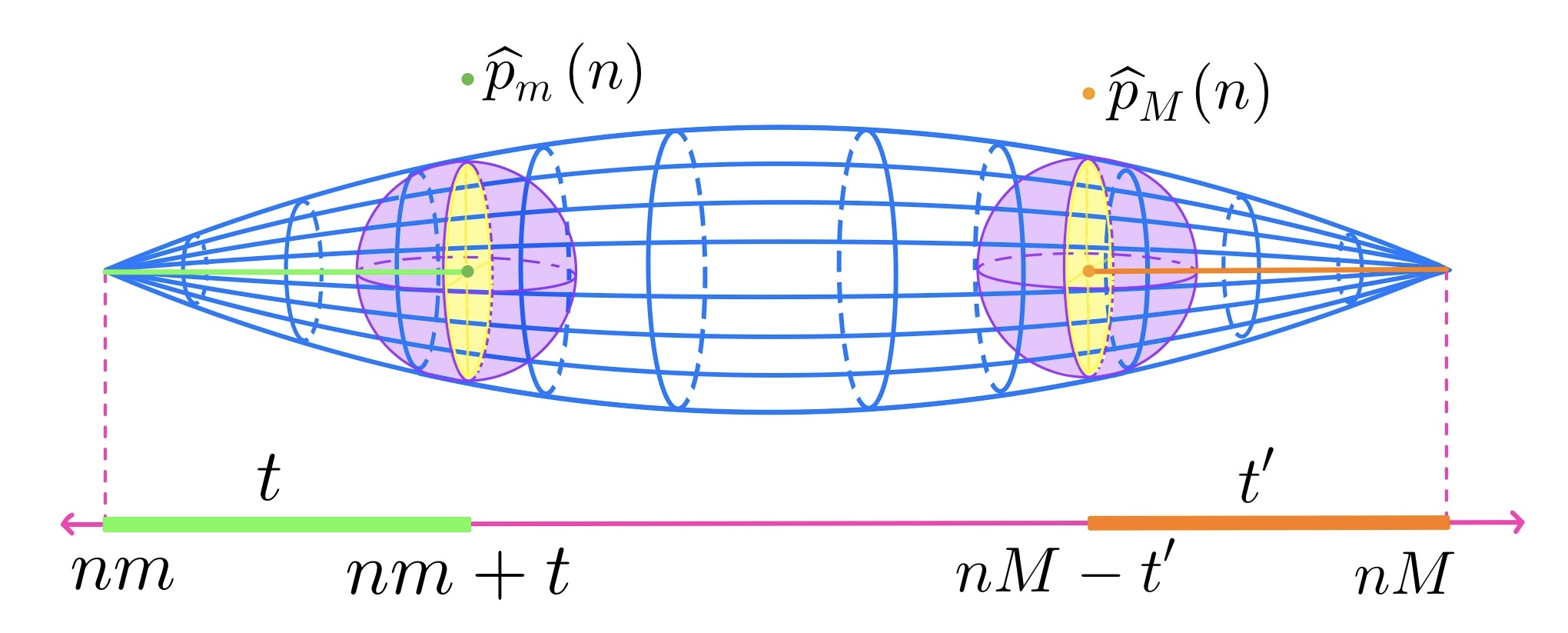}
    \caption{Illustration of the proof of Theorem~\ref{thm: AR structure}.}
    \label{fig:AAP_Proof_Sketch}
\end{figure}


\section{The arithmetic width of rational polytopes}

In this section, we will build upon our results from Theorem~\ref{thm: AR structure} for our analysis of the arithmetic width of dilates of convex polytopes. We begin this section with definitions that explore the complement of the arithmetic range: the sets of gaps. From there, we compare gaps across subsequent dilates of polytopes, which gives enough tools to prove Theorems~\ref{thm: eql in fixed directions} and~\ref{thm: eql_non_fixed_directions}. In what follows, we say $P$ is a \textit{rational polytope} if $P=\convhull(v_1,\dots,v_k)\subset\RR^d$ where $k$ is finite and each $v_i\in\QQ^d$ for $i\in [k]$. If each $v_i\in\ZZ^d$, we say that $P$ is \textit{integral}. The \textit{denominator of $P$}, denoted $\text{denom}(P)$, is the smallest positive dilation factor $n$ such that $nP$ is an integral polytope \cite{ziegler}.

\subsection{Arithmetic width in fixed directions}
In the previous section, we found that gaps of the arithmetic range of large dilates of a convex body occur near the extremes, motivating us to examine the behavior of the gaps as we dilate a polytope.

\begin{lemma}
\label{lemma: gap set equality}
Let $P$ be a rational polytope in $\mathbb{R}^d$ with denominator $D$, and fix a linear functional $\ell_c:\ZZ^d \to \ZZ$. Then, for any $n\gg 0$ for which $\aff(nP)$ contains a lattice point,  
\[|\gaps(AR_{c}(nP))| = |\gaps(AR_{c}((n+D)P))|.\]
\end{lemma}

\begin{proof}
Let $m = \min_{x \in P} \ell_c(x)$ and $M = \max_{x \in P} \ell_c(x)$. 
For sufficiently large $n$, Theorem~\ref{thm: AR structure} provides constants $t, t' > 0$ and step size $\lambda$ such that 
\[
\AR_c(nP) \cap [nm + t', nM - t]
\]
is a gap-free arithmetic progression of step size $\lambda$, and any gaps in $\AR_c(nP)$ may only occur in the regions $[nm, nm + t']$ or $[nM - t, nM]$.

Define the region near the minimum as
\[
P_{m, n} = nP \cap \ell_c^{-1}([nm, nm + t'])
\]
and for $(n+D)P$ as
\[
P_{m, n+D} = (n+D)P \cap \ell_c^{-1}([(n+D)m, (n+D)m + t']).
\]

Let $p_m \in P$ satisfy $\ell_c(p_m) = m$ and $p_M\in P$ satisfy $\ell_c(p_M) = M$. In the proof of Proposition~\ref{prop: aw_properties}\eqref{prop: aw_properties; translation}, we established that for any integral vector $v$ and convex set $K \subset \mathbb{R}^d$, we have $\AR_c(K + v) = \AR_c(K) + \ell_c(v)$. Thus, they will also have translated gap sets. In particular, since $P$ is a rational polytope, translation by $Dp_m$ or $Dp_M$ will be integer translations. As this preserves lattice points, we have that,
\[
P_{m, n+D} = P_{m, n} + Dp_m \qquad \text{and} \qquad P_{M, n+D} = P_{M, n} + Dp_M,
\]
for $n \gg 0$, so
\[
\AR_c(P_{m, n+D}) = \AR_c(P_{m, n}) + Dm \qquad \text{and} \qquad \AR_c(P_{M, n+D}) = \AR_c(P_{M, n}) + DM.
\]
Thus, we see that
\begin{align*}
\gaps_l(\AR_c((n+D)P)) &= \gaps_l(\AR_c(nP)) + Dm
\qquad \text{and}
\\
\gaps_r(\AR_c((n+D)P)) &= \gaps_r(\AR_c(nP)) + DM.
\end{align*}
Since $\gaps(\AR_c(P)) = \gaps_l(\AR_c(P)) \cup \gaps_r(\AR_c(P))$, the proof is now complete.  
\end{proof}

Having proved Lemma~\ref{lemma: gap set equality}, we now have the tools to proceed with the proof of Theorem~\ref{thm: eql in fixed directions}.  The proof also identifies the linear coefficient of the resulting quasilinear function.  

\begin{theorem}
\label{thm: eql in fixed directions}
Let $P$ be a rational polytope in $\mathbb{R}^d$. Then, for a fixed direction $c \in \mathbb{Z}^d$, the arithmetic width $\aw_c(nP)$ is an eventually quasilinear function of $n \in \mathbb{Z}_{\geq 0}$.
\end{theorem}

\begin{proof}
Let $m = \min_{x \in P} \ell_c(x)$ and $M = \max_{x \in P} \ell_c(x)$. Let $\lambda$ be the step size given by Theorem~\ref{thm: AR structure}. Lastly, let $D = \mathrm{denom}(P)$. We will prove that $\aw_c(nP)$ is eventually quasilinear by proving that
\begin{equation}
\label{eql:fixed-directions}
    \aw_c((n+D)P) = \aw_c(nP) + \tfrac{1}{\lambda} D(M - m)
\end{equation}
for all $n \in \ZZ_{\ge 0}$ for which $\aff(nP)$ contains a lattice point.  
Fix $k \in \{0, 1, \dots, D-1\}$ for which $\aff(kP)$ contains a lattice point.  Fix $t_k \in \mathbb{R}_{\geq 0}$ such that
\[
\max \AR_c(kP) = \max_{x \in kP} c^T x - t_k = kM - t_k.
\]
Then, for all $n \equiv k \bmod D$ with $n > k$, we have
\[
\max \AR_c(nP) = nM - t_k.
\]
Similarly, fix $t'_k \in \mathbb{R}_{\geq 0}$ such that for all $n \in \mathbb{Z}_{>0}$ with $n \equiv k \bmod D$,
\[
\min \AR_c(nP) = n m + t'_k.
\]
As such, if $n \equiv k \bmod D$, then
\[
\aw_c(nP)
= \tfrac{1}{\lambda}\big((nM - t_k) - (n m + t'_k) + 1\big) - |\gaps(\AR_c(nP))|.
\]
By an analogous argument, we obtain
\[
\aw_c((n+D)P) = \tfrac{1}{\lambda}\big(((n+D)M - t_k) - ((n+D)m + t'_k) + 1\big) - |\gaps(\AR_c((n+D)P))|.
\]
Applying Lemma~\ref{lemma: gap set equality}, we obtain
\[
\aw_c((n+D)P) - \aw_c(nP)
= \tfrac{1}{\lambda}D(M - m).
\]
This proves equality~\eqref{eql:fixed-directions}, so we conclude $\aw_c(nP)$ is indeed eventually quasilinear in $n$.
\end{proof}

This theorem has implications in discrete optimization. Integer linear programs (ILPs) are concerned with finding integral solutions to maximizing or minimizing linear functionals over polyhedra. As these problems are often more difficult than their real linear programming counterparts (LPs), to approximate solutions to ILPs, one often first solves the real relaxation and then rounds. The additive integrality gap is a way of measuring the accuracy of this approximation method~\cite{hosten2007computing, eisenbrand2008parametric}. Here we present a different construction.  

\begin{definition}
    Let $P$ be a rational polytope in $\mathbb{R}^d$ and fix linear functional $\ell_c$ for $c \in \mathbb{Z}^d$. We define the (additive) maximum and minimum integrality gaps as \[I_M(P,c)= \max_{x\in P}\ell_c(x) - \max_{x\in P\cap\ZZ^d}\ell_c(x) 
    \qquad\text{and} \qquad 
    I_m(P,c)= \min_{x\in P\cap\ZZ^d}\ell_c(x) - \max_{x\in P}\ell_c(x).\]   
\end{definition}

\begin{corollary}
\label{cor: integrality gap}
Let $P$ be a rational polytope in $\mathbb{R}^d$ and fix linear functional $\ell_c$ for $c \in \mathbb{Z}^d$. Then, for integers $t_1, t_2 \gg 0$ such that $t_1\equiv t_2\bmod D$ we have that \[I_M(t_1P, c) = I_M(t_2P, c) \qquad\text{and}\qquad I_m(t_1P, c) = I_m(t_2P, c).\] 
\end{corollary}

This follows immediately from Theorem~\ref{thm: eql in fixed directions}, since for each congruence class $t_1 \equiv t_2 \bmod D$, we have that $I_M(t_1P, c) = I_M(t_2P, c) = t_k$ and $I_m(t_1P, c) = I_m(t_2P, c) = t_k'$ where $t_k, t_k'$ are given in the proof of Theorem~\ref{thm: eql in fixed directions}. Thus, the integrality gap is eventually a periodic sequence (that is, quasi-constant) with period dividing $D$.

Corollary~\ref{cor: integrality gap} is closely related to themes in parametric integer programming. Earlier work by Hosten and Sturmfels~\cite{hosten2007computing} and  by Eisenbrand and Shmomin~\cite{eisenbrand2008parametric} studies the maximal additive integrality gap as you vary the data of the integer program, showing that it is always finite and computable in fixed dimension. This is often framed as varying the right-hand side of the constraints, which is equivalent to individually translating the facet-defining hyperplanes of a polytope. Dilating a polytope is one way to vary the constraints of a linear program. Their results establish that, in their more general framework, the additive integrality gap is eventually quasipolynomial. 

\begin{remark}
\label{rmk: structure theorem for sets of length}
Returning to the setting of Examples~\ref{ex: any finite set is an arithmetic range} and~\ref{ex: almost arithmetic progression}, one of the cornerstones of factorization theory is the \emph{structure theorem for sets of length}, which states that for certain families of semigroups, the set of factorization lengths of large elements are almost arithmetic sequences~\cite{structuretheorem,setsoflengthsurvey}.  For numerical semigroups specifically, the sets of length of large elements were recently shown in~\cite{structuretheoremns} to satisfy a periodicity reminiscent of Theorem~\ref{thm: eql in fixed directions}.  In fact, applying Lemma~\ref{lemma: gap set equality}, and the main ideas of its proof, to the numerical semigroup polytope~$P_S$ from Example~\ref{ex: any finite set is an arithmetic range} yields an alternative proof of their main result \cite[Theorem~4.2]{structuretheoremns}.  
\end{remark}

\subsection{Arithmetic width in minimizing directions}

In order to use Theorem~\ref{thm: eql in fixed directions} to prove the same eventually quasilinear behavior of the arithmetic width, we first examine the minimizing directions across varying dilates of our polytope. This allows us to prove our second main result, Theorem~\ref{thm: eql_non_fixed_directions}.

\begin{lemma}
\label{lemma: optimal mod directions}
    Let $P\subset\RR^d$ be a rational polytope with denominator $D$. If $n_1\equiv n_2\bmod D$, then $\aw(n_1P)$ and $\aw(n_2P)$ are obtained in the same directions for $n_1, n_2\gg 0$.
\end{lemma}

\begin{proof}
    Without loss of generality, assume $n_1>n_2$ such that $n_1, n_2\gg0$, and fix $k\in\ZZ_{>0}$ such that $n_1=n_2+kD$. Then, applying the recurrence relationship in Equation~\eqref{eql:fixed-directions}, we see 
    \[
    \aw_c(n_2+(k+1)D)P = \aw_c(n_2P)+\tfrac{1}{\lambda}D(k+1)(M-m).
    \]
    As we also have that
    \[
    \aw_c(n_1+D)P = \aw_c(n_1P)+ \tfrac{1}{\lambda}D(M-m),
    \]
    setting these expressions equal gives us that 
    \begin{equation}
    \label{mod dil}
        \aw_c(n_1P)-\aw_c(n_2P)= \tfrac{1}{\lambda}kD(M-m). 
    \end{equation}
    Assume $c_*\in\ZZ^d\setminus\{0\}$ is such that $\aw(n_1P)=\aw_{c_*}(n_1P)$. By way of contradiction, assume $\aw(n_2P)\neq\aw_{c_*}(n_2P)$, so $\aw(n_2P)=\aw_{d_*}(n_2P)<\aw_{c_*}(n_2P)$ for some $d_*\in\ZZ^d\setminus\{0\}$. However, adding $\tfrac{1}{\lambda}kD(M-m)$ to both sides of the inequality and applying Equation~\eqref{mod dil} yields $\aw_{d_*}(n_1P)<\aw_{c_*}(n_1P)$ which is a contradiction. This proves that if $\aw(n_1P)=\aw_{c_*}(n_1P)$, then $\aw(n_2P)=\aw_{c_*}(n_2P)$. 

    Conversely, assume $c_*\in\ZZ^d\setminus\{0\}$ is such that $\aw(n_2P)=\aw_{c_*}(n_2P)$. As $n_2=n_1-kD$, we rewrite Equation~\eqref{eql:fixed-directions} as 
    \[
    \aw_c(n_2P)=\aw_c(n_1P) - \tfrac{1}{\lambda}kD(M-m).
    \]
    By way of contradiction, assume $\aw(n_1P)\neq\aw_{c_*}(n_1P)$. Thus, there is some $d_*\in\ZZ^d\setminus\{0\}$ such that $\aw(n_1P)=\aw_{d_*}(n_1P)<\aw_{c_*}(n_1P)$. However, subtracting $\tfrac{1}{\lambda}kD(M-m)$ from both sides of this inequality yields $\aw_{d_*}(n_2P)<\aw_{c_*}(n_2P)$ which gives us the same contradiction. Therefore, when $n_1\equiv n_2\bmod D$, we have that $\aw(n_1P)$ and $\aw(n_2P)$ are obtained in the same directions for $n_1, n_2\gg 0$. 
\end{proof}

\begin{figure}[t]
    \centering
    \includegraphics[width=.8\linewidth]{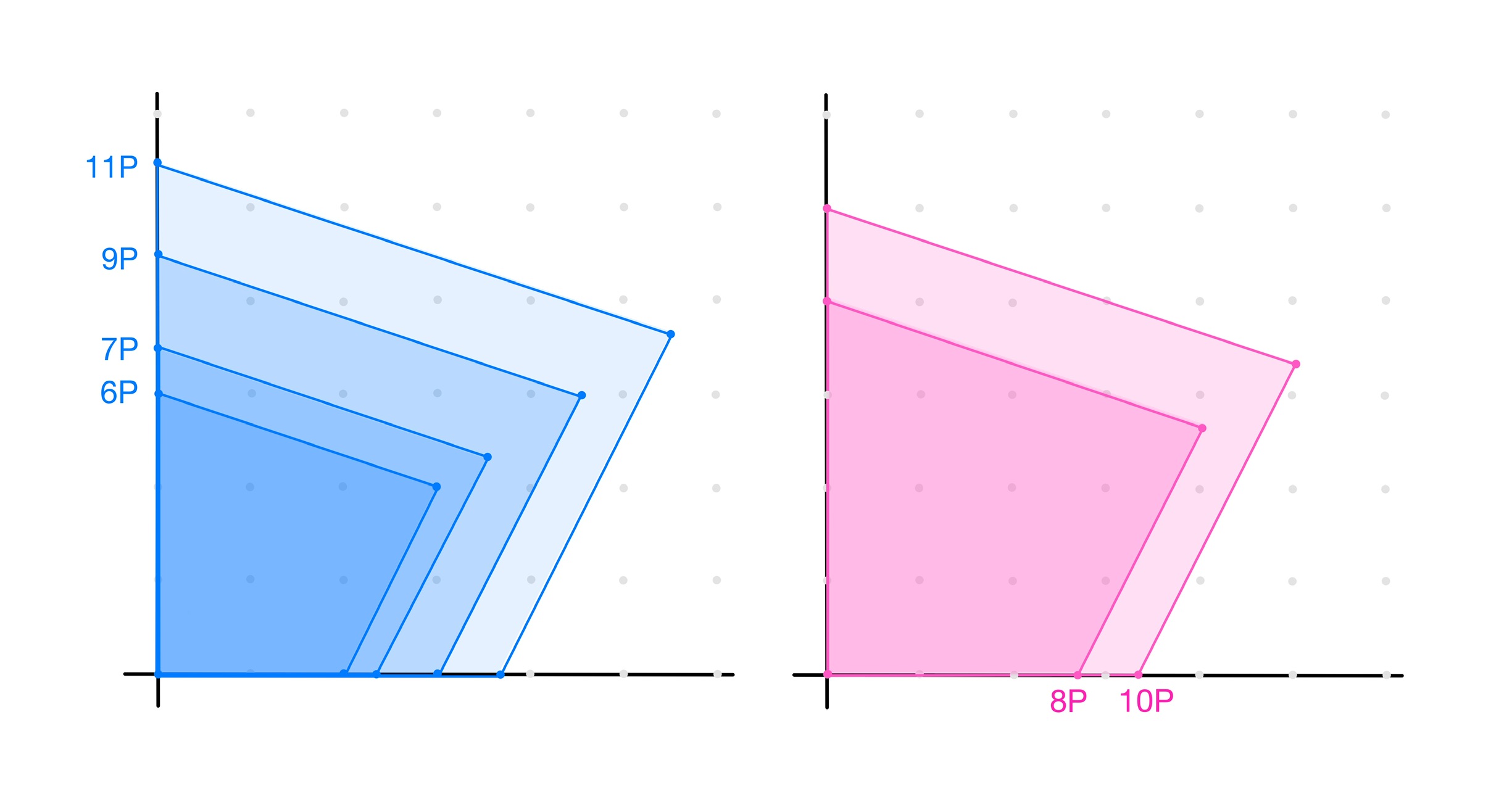}
    \caption{Dilations of the polytope $P$ from Example~\ref{ex: periodic directions}, partitioned based on their optimal directions.}
    \label{fig:periodic_directions}
\end{figure}

\begin{example}
    \label{ex: periodic directions}
    Consider the polytope 
    \[
    P=\convhull\big(\big(0,0),\big(\tfrac{1}{3}, 0\big), \big(0, \tfrac{1}{2}\big), \big(\tfrac{1}{2}, \tfrac{1}{3}\big)\big).
    \]
    The optimal directions are oscillating periodically. In Figure~\ref{fig:periodic_directions}, the polytopes given in the right image have arithmetic width minimized in the same directions. Similarly, the polytopes given in the left image have their arithmetic widths minimized in the same directions. More specifically, let $c_n$ be a direction obtaining the arithmetic width of $nP$. Then,
    \[c_n = \begin{cases} {\color{blue}(1,0) \text{ or } (0,1)} &\text{  for } n\equiv 0,1,3,5 \bmod 6, \\  {\color{magenta}(0,1)} &\text{  for } n\equiv 2,4 \bmod 6 .  
    \end{cases}\]
    As ensured by Lemma~\ref{lemma: optimal mod directions}, when $n$ is restricted to the same residue class modulo 6, the polytopes $nP$ have the same optimal directions. 
    \end{example}

\begin{theorem}
\label{thm: eql_non_fixed_directions}
    Let $P\subset\RR^d$ be a rational polytope. The arithmetic width $aw_P(n)$ is eventually quasilinear function of $n$ for $n\in\ZZ_{\geq 0}$.
\end{theorem}

\begin{proof}
    Recall, from the proof of Theorem~\ref{thm: eql in fixed directions}, we have that \[aw_c((n+D)P) = aw_c(nP) + \tfrac{1}{\lambda}D(M_c-m_c),\] where $M_c=\max_{x\in P} c^Tx$ and $m_c=\min_{x\in P} c^Tx$. From Lemma~\ref{lemma: optimal mod directions}, it suffices to consider a set $D$ of optimal directions to describe the behavior of $\aw(nP)$. Let $c_i^*$ be an optimal direction for all $t_iP$ where $t_i \equiv i\bmod D$ for $i=0,1,\dots, D-1$. 
    Letting $M_i=M_{c_i^*}$ and $m_i=m_{c_i^*}$ for each $i$ for ease of display, this yields 
    \begin{equation*}
    \aw(nP)=
    \begin{cases}
        \tfrac{1}{\lambda}(M_0-m_0)n + a_0 & \text{ for } n\equiv 0\bmod D,\\
        \tfrac{1}{\lambda}(M_1-m_1)n + a_1 & \text{ for } n\equiv 1\bmod D,\\
        \hspace{10mm}\vdots  & \hspace{10mm}\vdots\\
        \tfrac{1}{\lambda}(M_{D-1}-m_{D-1})n + a_{D-1} & \text{ for } n\equiv (D-1)\bmod D,\\
    \end{cases}
    \end{equation*}
    for some $a_0, a_1, \ldots, a_{D-1}$, which proves our claim.   
\end{proof}

\section{Computing arithmetic width:\ algorithms and complexity}


How do we compute the arithmetic width? Here, we will use an algebraic technique wherein we encode lattice points as exponent vectors of monomials, the rational generating functions of the lattice points in a polytope. From there, we will rely on Barvinok's theory for the efficient encoding of the lattice points of polyhedra as short sums of rational functions. For details on the theory and implementations of Barvinok's theorem, see \cite{barvinokpommersheim, MR1992831, DELOERA20041273}. The results that follow assume a fixed dimension $d$, and all stated running times are in the input bit-size of the rational polytope. For rational polytopes, we provide a polynomial time algorithm to compute the arithmetic width in fixed direction and a singly-exponential time algorithm to compute the arithmetic width. We will now state the key facts needed for our main results.

Given a rational polyhedron $P$, the associated generating function of its lattice points is the Laurent series 
\[
g_P(x) = \!\!\! \sum_{\alpha\in P\cap\ZZ^d} \!\!\! x^{\alpha}.
\] Barvinok's theory shows that, in fixed dimension $d$, this (generally exponential) sum admits a polynomial-size representation as a sum of rational functions of the form 
\begin{equation} \label{eq:aa}
g_P(x) = \sum_{i\in I} {E_i \frac{x^{u_i}} {\prod\limits_{j=1}^d
(1-x^{v_{ij}})}},
\end{equation}
where $I$ is a polynomial-size indexing set, $E_i\in\{1,-1\}$ and $u_i, v_{ij} \in\ZZ^d$ for all $i$ and $j$. 

The key idea is that Barvinok's compact representation of the lattice points supports efficient computations. For example, substituting $x_i=1$ for all $i\in[d]$ yields the number of lattice points in $P$. More generally, the following lemma shows us that such substitutions can be done in polynomial time. This tool will later be used to compute the arithmetic range and arithmetic width. 

\begin{lemma}[{\cite[Theorem~2.6]{MR1992831}}] \label{monosubs} 
    Let 
    \[
    g(x)= \sum_{i\in I}\alpha_i \frac{x^{u_i}}{\prod_{j=1}^k(1-x^{v_{ij}})},
    \]
    where $u_i,v_{ij}\in\ZZ^d$, $\alpha_i \in \QQ$, and $k$ is fixed. Let $\psi:\CC^n\to\CC^d$ be the monomial map defined by $x_i\mapsto z_1^{l_{i1}}z_2^{l_{i2}}\dots z_n^{l_{in}},$ and assume $\psi(\CC^n)$ is not contained in the pole set of $g(x)$. Then, $g(\psi(z))$ can be computed in polynomial time (in the input bit size) as a short rational generating function of the same form as $g(z)$.
\end{lemma}

Having obtained a compact rational generating function for the lattice points of $P$, the next step is to compute the arithmetic width in a fixed direction is a projection step. This is achieved using the Projection Theorem of Barvinok and Woods~\cite{MR1992831}. 
   
\begin{lemma}[{\cite[Theorem~1.7]{MR1992831}}] \label{project}       
    Fix the dimension $d$. Let $P\subset\RR^d$ be a rational polytope and let $T:\ZZ^d\to\ZZ^k$ be a linear map. Then, there is a polynomial-time algorithm that computes a short rational generating function for $f\bigl(T(P\cap\ZZ^d);z\bigr)$. 
\end{lemma}

Having gathered the appropriate machinery, we are now prepared to prove that $\aw_c(P)$ can be computed in polynomial time when the dimension is fixed. 

\begin{theorem}
\label{thm: computing aw in fixed directions}
    Let $P\subset\RR^d$ be a rational polytope and let $L$ be the input bit-size of $P$. In fixed dimension, given a nonzero integer vector $c\in\ZZ^d\setminus\{0\}$, there is an algorithm to compute $\aw_c(P)$ that is polynomial time in~$L$ and the input bit-size of $c$.
\end{theorem}

\begin{proof}
    By Barvinok's algorithm, the generating function $g_P(x)=\sum_{\alpha\in P\cap\ZZ^d}x^{\alpha}$ can be computed in polynomial time and expressed in the short rational form of Equation~\eqref{eq:aa}. Applying Lemma~\ref{monosubs} to the linear map $\ell_c:\ZZ^d\to\ZZ$ defined $\ell_c(x)=c^Tx$, we obtain a univariate rational generating function $h(t)$ whose monomials correspond bijectively to $\AR_c(P)$. Note that $h$ is computed in compact rational form, not as an explicit list of monomials. Finally, the arithmetic width in direction $c$ is the number of distinct monomials in $h$. By Lemma~\ref{project}, this quantity can be determined in polynomial-time. Evaluating $h(t)$ at $t=1$ yields precisely the desired count. 
\end{proof}

In order to extend the results of Theorem~\ref{thm: computing aw in fixed directions} to the minimum arithmetic width, we first exhibit a finite set of test directions that it suffices to check to compute the arithmetic width. 

\begin{lemma}
\label{lemma: sufficent test directions}
     Let $P\subset\RR^d$ be a rational polytope and fix a nonzero integer vector $c\in\ZZ^d\setminus\{0\}$. There is a finite set of directions $T$ that serve as a sufficient test set to compute $\aw(P)$.
\end{lemma}

\begin{proof}
    Without loss of generality, we may assume that $P$ contains at least two integer points and that $\dim(P)=d$. If $c$ is a direction that minimizes the arithmetic width, then there are at least two distinct lattice points $x,y\in K\cap\ZZ^d$ such that $\ell_c(x)=\ell_c(y)$, i.e., $c^T(x-y)=0$. Thus, it suffices to consider directions that collapse differences of lattice points in $K$. Thus, to determine $T$, it suffices to consider directions $c$ that are orthogonal to nonzero differences of lattice points in $K$. We may use this to determine $T$ as follows:
    \begin{enumerate}[1.]
        \item Compute $N=(P-P)\cap(\ZZ^d)$. For rational polytopes, $|N|$ can be computed in polynomial time in fixed dimension using methods such as Barkinov's algorithm. However, the number of lattice point in $P$, and thus $N$, can be exponentially large in the bit-length $L$ of the input data, even for fixed $d$. 
        \item Enumerate all linearly independent subsets of $N$ of size $d-1$. For each such subset $T'=\{v_1,\dots, v_{d-1}\}$, form the matrix $V=(v_1|\dots|v_{d-1})$.
        \item For each such $V$, compute a primitive integer vector $c\in\ker(V^T)$. Add each $c$ to $T$. \qedhere
    \end{enumerate}
\end{proof}

\begin{theorem}\label{thm: computing aw}
    Let $P\subset\RR^d$ be a rational polytope with input bit-size $L$. In fixed dimension, given a nonzero integer vector $c\in\ZZ^d\setminus\{0\}$, there is a single-exponential time algorithm in $L$ to compute the arithmetic width of $P$.
\end{theorem}

\begin{proof}
    Using Lemma~\ref{lemma: sufficent test directions} to compute $T$ and then apply the algorithm from Theorem~\ref{thm: computing aw in fixed directions} to each $t\in T$ to determine the minimizing direction and value. 
\end{proof}

With this, we have proven our third main result.
\begin{proof}[Proof of Main Theorem~\ref{mainthm:awu-polytime}]
Theorems~\ref{thm: computing aw in fixed directions} and~\ref{thm: computing aw} and Lemma~\ref{lemma: sufficent test directions} prove this claim. 
\end{proof}

\section{Acknowledgements} We are grateful to János Pach, Imre Bárány, Daniel Dadush for fruitful discussions, and to Bruce Reznick for providing important references. This research was partially supported by NSF grants 2348578 and 2434665.

\bibliographystyle{plain}
\bibliography{aw}


\end{document}